\documentclass{amsart}
\usepackage{latexsym,amssymb,amsmath,amsthm,latexsym}
\usepackage{amsfonts}
\theoremstyle{definition}

\newtheorem{example}{Example}
\newtheorem{corollary}{Corollary}

\newtheorem{proposition}{Proposition}

\newtheorem{remark}{Remark}

\newtheorem{identity}{Identity}
\title{Enumeration of Restricted Words and Linear Recurrence Equations}
\author{Milan Janji\'c}
\date{\today}
\begin{document}
\maketitle
\begin{center}Department for Mathematics and Informatics, University of Banja
Luka,\end{center}
 \begin{center}Republic of Srpska, BA\end{center}

\begin{abstract}
 In previous papers, for an arithmetical function $f_0$, we defined functions $f_m$ and $c_m$ and designated numbers of  restricted words over a finite alphabet counted by these functions. In this paper, we examine the reverse problem for five specific types of restricted words. Namely, we find the initial function $f_0$ such that $f_m$ and $c_m$ enumerate these words. In each case, we derive explicit formulas for $f_m$ and $c_m$.

 Fibonacci, Merssen, Pell,  Jacosthal, Tribonacci, and Padovan numbers all appear as values of $f_m$, so we obtain new formulas for these numbers. Also, we
  combinatorially  derive  explicit formulas for the solutions of five types of homogenous linear recurrence equations.
\end{abstract}

\section{Introduction} This paper is a continuation of the investigations of
the problem of restricted words enumeration from the author's  previous papers~\cite{ja1,ja2,ja3},
where two functions $f_m$ and $c_m$ were defined as follows. For an initial arithmetic function $f_0$, the function $f_m,(m\geq 1)$ is the $m^\text{th}$ invert transform of $f_0$. The function $c_m(n,k)$ was defined as
\begin{equation}\label{cmnk}
c_m(n,k)=\sum_{i_1+i_2+\cdots+i_k=n}f_{m-1}(i_1)\cdot f_{m-1}(i_2)\cdots f_{m-1}(i_k),\end{equation}
where the sum is  over positive $i_1,i_2,\ldots,i_k$.
For $m\geq 1$, the following formula holds:
\begin{equation}\label{suma1}f_m(n)=\sum_{k=1}^nc_m(n,k).\end{equation}
The functions $f_m$ and $c_m$ depend only on the initial function $f_0$, and are related to the enumeration of weighted compositions. Namely, if weights are $\{f_{m-1},f_{m-2}(2),\ldots\}$, then $f_m(n)$ is the number of all weighted compositions of $n$, and $c_m(n,k)$ is the number of weighted compositions of $n$ into $k$ parts.

In Janji\'c~\cite{ja1,ja2,ja3}, weighted compositions were related to
restricted words over a finite alphabet. For a given initial function $f_0$, we investigated restricted words counted by $f_m$ and $c_m$.   In this paper, we reverse the problem. Namely, for a particular type of restricted words, we first find the initial function $f_0$ which count such words. We then derive formulas for $f_m$ and $c_m$ and give its combinatorial meanings in terms of restricted words.

We restate~\cite[Propositions 12]{ja3}, which will be used frequently in the paper.
\begin{proposition}\label{alf}
 Assume that $f_{0}(1)=1$ and $m>1$. Assume next  that, for $n\geq 1$, we have $f_{m-1}(n)$
words of length $n-1$ over a finite alphabet $\alpha$. Let  $x$ be a letter which is not in $\alpha$. Then, $c_m(n,k)$ is the number of words of length $n-1$ over the alphabet $\alpha\cup\{x\}$ in which $x$ appears exactly $k-1$ times.
\end{proposition}
We also restate the result in~\cite[Proposition 6]{ja3}.
The following formula holds:
\begin{equation}\label{cm1}c_m(n,k)=\sum_{i=k}^n(m-1)^{i-k}{i-1\choose k-1}c_1(n,i),\;(1\leq k\leq n).\end{equation}
 We consider the following five types of restricted words:
\begin{itemize}
\item[1.]
Words over the alphabet $\{0,1,\ldots,a-1,\ldots,m+a-1\}$,
such that no two adjacent letters from $\{0,1,\ldots,a-1\}$ are the same.
\item[2.]
Words over the alphabet
$\{0,1,\ldots,a-1,\ldots,a+m-1\}$ such that letters $0,1,\ldots,a-1$
avoid a run of odd length.
\item[3.] Words over the alphabet $\{0,1,\ldots,b-1,\ldots,m+a-1\}$
  avoiding subwords of the form $0i,(i=1,\ldots,b)$.
 \item[4.]
 Words over the  alphabet $\{0,1,\ldots,m+1\}$ such that $0$ and $1$ appear only as  subwords of the form $1i$, where $i$ is a run of zeros of length at least $1$.
 \item[5.]
 Words over the  alphabet $\{0,1,\ldots,m+1\}$ in which $0$ appears only in a run of even length, and $1$ appears only in a run the length of which is divisible by $3$.
\end{itemize}
We note that the initial function $f_0$ is defined by a linear homogenous recurrence in all cases.
 \section{Case 1} To solve the problem posed in Case 1, we  consider the following linear recurrence:
\begin{equation*}f_0(1)=1,f_0(a)=a,f_0(n+2)=(a-1)f_0(n+1),(n\geq 1),
\end{equation*}
where $a>0$.
It is easy to see that
\begin{equation*}f_0(n)=a(a-1)^{n-2},(n\geq 2).\end{equation*}
\begin{remark}
This formula appears in Birmajer at al.~\cite[Example 17]{bir}. Also, the case $a=1$ is considered in~\cite[Example 18]{ja3}.
\end{remark}
The function $f_0$ has the following combinatorial interpretation:
\begin{proposition} The number $f_0(n)$ is the number of words of length $n-1$ over the alphabet $\{0,1,\ldots,a-1\}$  such that no two adjacent letters are the same.
\end{proposition}
\begin{proof} We have $f_0(1)=1$, since only the empty word has length $0$.
 Also, $f_0(2)=a$, since a word of length $1$ may consist of an arbitrary letter.
 To obtain a word of length $n+2$, for $n>0$, we need to insert $a-1$ letters in front of each word of length $n+1$.
\end{proof}
As an immediate consequence of Janji\'c~\cite[Corollary 9]{ja1}, we obtain
\begin{corollary}\label{co1}
For $m\geq 0$, the following recurrence holds:
\begin{gather*}f_m(1)=1,f_m(2)=m+a,\\f_m(n+2)=(m+a-1)f_m(n+1)+mf_m(n),(n\geq 1).\end{gather*}
\end{corollary}
We next prove that $f_m$ counts the desired words.
\begin{proposition}
The number $f_m(n)$ is the number  of words of length $n-1$ over
the alphabet $\{0,1,\ldots,a-1,a,\ldots,m+a-1\}$,
such that no two adjacent letters from $\{0,1,\ldots,a-1\}$ are the same.
\end{proposition}
\begin{proof}
 We have $f_m(1)=1$, since only the empty word has length $0$. Also, $f_m(2)=m+a$
 since a word of length $1$ may consist of any letter of the alphabet.
Assume that $n>2$. Consider a word of length $n+1$. In front of such a word, we insert a letter different from the first letter of the word. In this way, we obtain all words of length $n+2$ beginning with two different letters. The remaining words must begin with two same letters. Since there are $mf_m(n)$ such words, the statement is true.
\end{proof}
\begin{remark} The continued fraction $[f_0(1);f_0(2),f_0(3),\ldots]$ equals $\sqrt 2$. Also,
the sequence $f_1(1),f_1(2),\ldots,f_1(n)$ is the numerator of the $n$th convergent  of $\sqrt 2$.
\end{remark}
Since $f_m(1)=1$, we may apply Proposition \ref{alf} to obtain
\begin{corollary} The number $c_m(n,k)$ is the number of words of length $n-1$ over
$\{0,1,\ldots,a-1,\ldots,m+a-1\}$ in which $k-1$ letters equal $m+a-1$, and no two letters from $\{0,1,\ldots,a-1\}$ are identical.
\end{corollary}
We next derive an explicit formula for $c_1(n,k)$.
\begin{proposition} We have
\begin{equation}\label{c1nk}c_1(n,n)=1,c_1(n,k)=\sum_{i=0}^{k-1}{k\choose i}{n-k-1\choose k-i-1}a^{k-i}(a-1)^{n-2k+i},(k<n).
\end{equation}
\end{proposition}
\begin{proof}
From (\ref{cmnk}), we firstly obtain $c_1(n,n)=1$. Assume that $k<n$. Since at most $k-1$ of $i_t,(t=1,2,\ldots,k)$ may equal $1$, then
\begin{gather*}c_1(n,k)=\sum_{i=0}^{k-1}{k\choose i}\sum_{j_1+j_2+\cdots+j_{k-i}=n-i}f(j_1)f(j_2)\cdots f(j_{k-i})\\
=\sum_{i=0}^{k-1}{k\choose i}a^{k-i}(a-1)^{n-2j+i}\sum_{j_1+j_2+\cdots+j_{k-i}=n-i}1\\
=\sum_{i=0}^{k-1}a^{k-i}(a-1)^{n-2k+i}{k\choose i}{n-k-1\choose k-i-1}.
\end{gather*}
\end{proof}
\begin{remark} Note that, in (\ref{c1nk}), terms in which  $i<2k-n$ would equal zero.
\end{remark}
  To obtain an explicit formula for $c_m(n,k)$, we use (\ref{cm1}).
  We first extract the term for $i=n$ to obtain
  \begin{equation*}c_m(n,k)=m^{n-k}{n-1\choose k-1}+\sum_{i=k}^{n-1}{i-1\choose k-1}c_1(n,i).\end{equation*}
It follows that
  \begin{equation*}c_m(n,k)=m^{n-k}{n-1\choose k-1}+\sum_{i=k}^{n-1}\sum_{j=0}^{i-1}(m-1)^{i-k}a^{i-j}(a-1)^{n-2i+j}{i-1\choose k-1}{n-i-1\choose i-j-1}{i\choose j}.
  \end{equation*}
Using (\ref{suma1}), we obtain the following formula for $f_m(n)$:
\begin{gather*}
f_m(n)=m^{n-1}+\sum_{k=1}^{n-1}\sum_{i=k}^{n-1}\sum_{j=0}^{i-1}(m-1)^{i-k}a^{i-j}(a-1)^{n-2i+j}{i-1\choose k-1}{i\choose j}{n-i-1\choose i-j-1}.
\end{gather*}
\section{Case 2}
Let $a$ be a positive integer. Define $f_0$ as follows:
\begin{equation}\label{r1}f_0(1)=1,f_0(2)=0, f_0(n+2)=af_0(n),(n\geq
1).\end{equation}
 We firstly describe the restricted words counted by $f_0$.
\begin{proposition}\label{ll1} For $a>0$, the number $f_0(n)$ is the number of words of length $n-1$ over the alphabet $\{0,1,\ldots,a-1\}$ in which there are no runs of odd length.
\end{proposition}
\begin{proof}
Let $d(n)$ denote the number of words of length $n$, which we wish to
count. Firstly, $d(0)=1$ since only the empty word has length $0$. Next, $d(1)=0$ as there are no runs of length $1$.
Assume that $n>2$. A word of length $n$ must begin with two identical letters. Hence, there are $ad(n-2)$ such words.
We conclude that  the following recurrence holds:
\[d(0)=1,d(1)=0, d(n)=ad(n-2),(n\geq 2),\]
which yields  $d(n-1)=f_0(n),(n\geq 1)$.
\end{proof}
From (\ref{r1}), we easily obtain the following explicit formula for $f_0$:
\begin{equation*}
f_0(n)=\begin{cases}0,&\text{ if $n=2t$};\\a^t,&\text{ if $n=2t+1$}.
\end{cases}
\end{equation*}
\begin{corollary}\label{co2}
For $m\geq 0$, the following recurrence holds:
\begin{gather*}f_m(1)=1,f_m(2)=m,\\f_m(n+2)=mf_m(n+1)+af_m(n),(n\geq 1).\end{gather*}
\end{corollary}
\begin{proof}
The proof is a consequence of~\cite[Corollary 9]{ja1}.
\end{proof}
\begin{proposition}\label{pp5}
The number $f_m(n)$ is the number of words of length $n-1$ over the alphabet
$\{0,1,\ldots,a-1,\ldots,a+m-1\}$, such that letters $0,1,\ldots,a-1$
avoid runs of odd length.
\end{proposition}
\begin{proof}
We let $d(n)$ denote the number of desired words of length $n-1$. It is clear that $d(0)=1$ and $d(1)=m$.
 A word of length $n+1$ may begin with a letter from $\{a,a+1,\ldots,a+m-1\}$. There are $md(n)$ such word.
 If a word begins with a letter from $\{0,1,\ldots,a-1\}$,  it must be followed by the same letter.
  Hence, there are $ad(n-1)$ such words.
We conclude that $d(n)=f_m(n+1)$.
\end{proof}

  Some well-known classes of numbers satisfy the  recurrence from Corollary \ref{co2}. We give the appropriate combinatorial meaning for some of them.
\begin{enumerate}
\item The case  $a=1,m=1$ concerns the Fibonacci numbers.
The number of binary words of length $n-1$ in which $0$ avoids a run of odd length is
$F_n$.
\item The case  $a=1,m=2$ concerns the Pell numbers $P_n$. The number of ternary words of length $n-1$
 in which $0$ avoids runs of odd length is $P_n$.
  \item The case  $a=2,m=1$ concerns the Jacobhstal  numbers $J_n$.  The number of ternary words of
  length $n-1$ in which $0$ and $1$ avoid runs of odd length is $J_n$.
\end{enumerate}

From the combinatorial interpretation, we easily derive an explicit formula  for $f_m(n)$.
\begin{proposition}\label{lrj} We have
\[f_m(n)=\sum_{j=0}^{\lfloor\frac {n-1}{2}\rfloor}m^{n-2j-1}a^{j}{n-1-j\choose j}.\]
\end{proposition}
\begin{proof}According to Proposition \ref{pp5}, in a word counted by $f_m$, the letters from $\{0,1,\ldots,a-1\}$ may appear only in pairs. There are $a$ such pairs.
 We may choose $j,(0\leq j\leq \lfloor\frac{n-1}{2}\rfloor)$  pairs in a word of length $n-1$. These $j$ pairs may be chosen in ${n-j-1\choose j}$ ways. When we have chosen $j$ pairs from $\{0,1,\ldots,a-1\}$, the remaining $n-1-2j$ letters are from $\{a,a+1,\ldots,a+m-1\}$,  which are $m$ in number.
\end{proof}
As a consequence, we obtain the following similar explicit formulas for the Fibonacci, Pell and Jacobsthal numbers:
\begin{gather*}F_n=\sum_{j=0}^{\left\lfloor\frac{n-1}{2}\right\rfloor}{n-j-1\choose j}, P_n=\sum_{j=0}^{\left\lfloor\frac{n-1}{2}\right\rfloor}2^{n-2j-1}{n-j-1\choose j},\\J_n=\sum_{j=0}^{\left\lfloor\frac{n-1}{2}\right\rfloor}2^{j}{n-j-1\choose j}
.\end{gather*}
\begin{corollary}
The number $c_m(n,k)$ is the number of words of length $n-1$ over the alphabet $\{0,1,\ldots,a-1,\ldots,a+m-1\}$ in which the letter $a+m-1$  appears $k-1$ times, and letters from $\{0,1,\ldots,a-1\}$ avoid runs of odd length.
\end{corollary}
\begin{proof} The proof follows from Proposition \ref{alf}.
\end{proof}
We now derive an explicit formula  for $c_1(n,k)$.
\begin{proposition}\label{plr} The following equation holds:
\begin{equation*}c_1(n,k)=\begin{cases}a^{\frac{n-k}{2}}{\frac{n+k}{2}-1\choose k-1},&\text{ if $n-k$ is even};\\0,&\text{ if $n-k$ is odd }.\end{cases}
\end{equation*}
\end{proposition}
\begin{proof}
Each term in (\ref{cmnk}) in which $i_t$ is even equals zero. Hence, (\ref{cmnk}) becomes
\begin{gather*}c_1(n,k)=\sum_{2j_1+1+2j_2+1+\cdots+2j_k+1=n}a^{j_1}\cdot a^{j_2}\cdots a^{j_k}\\=a^{\frac{n-k}{2}}\sum_{j_1+j_2+\cdots+j_k=\frac{n+k}{2}}1=a^{\frac{n-k}{2}}{\frac{n+k}{2}-1\choose k-1}.
\end{gather*}
\end{proof}
As a consequence of (\ref{suma1}), we obtain the following explicit formulas for the Fibonacci and Jacobsthal numbers:
\begin{gather*}F_{2n}=\sum_{k=1}^n{n+k-1\choose n-k},F_{2n-1}=\sum_{k=1}^n{n+k-2\choose n-k},\\
J_{2n}=\sum_{k=1}^n2^{n-k}{n+k-1\choose n-k},J_{2n-1}=\sum_{k=1}^n2^{n-k}{n+k-2\choose n-k}.
\end{gather*}
Furthermore, we derive an explicit formula for $c_2(n,k)$.
Using (\ref{cm1}), for even $n$, we obtain
\begin{gather*}c_2(2n,k)=\sum_{i=k}^{2n}{i-1\choose k-1}c_1(2n,i)=
\sum_{j=\lceil\frac k2\rceil}^{n}{2j-1\choose k-1}c_1(2n,2j)\\
=\sum_{j=\lceil\frac k2\rceil}^{n}a^{n-j}{2j-1\choose k-1}{n+j-1\choose n-j}.
\end{gather*}
For odd $n$, we have
\begin{gather*}c_2(2n-1,k)=\sum_{i=k}^{2n-1}{i-1\choose k-1}c_1(2n,i)=
\sum_{j=\lceil\frac {k+1}{2}\rceil}^{n}{2j-2\choose k-1}c_1(2n-1,2j-1)\\
=\sum_{j=\lceil\frac {k+1}{2}\rceil}^{n}a^{n-j}{2j-2\choose k-1}{n+j-2\choose n-j}.
\end{gather*}
In particular, for $a=1$, we obtain the following formulas for Pell numbers:
\begin{gather*}P_{2n}=\sum_{k=1}^{2n}\sum_{j=\lceil\frac k2\rceil}^{n}{2j-1\choose k-1}{n+j-1\choose n-j},\\
P_{2n-1}=\sum_{k=1}^{2n-1}\sum_{j=\lceil\frac{k+1}{2}\rceil}^{n}{2j-2\choose k-1}{n+j-2\choose n-j}.
\end{gather*}
\begin{remark} Using (\ref{cm1}), we may obtain an explicit formula for $c_m(n,k)$.
\end{remark}
\section{Case 3}
 Let $a>b>0$ be integers. We define $f_0$ by the following recurrence:
\begin{equation}\label{re2}f_0(1)=1,f_0(2)=a, f_0(n+2)=af_0(n+1)-bf_0(n),(n\geq 1).\end{equation}
\begin{proposition}
 The number  $f_0(n)$
is the number of words of length $n-1$ over the alphabet $\{0,1,\ldots,a-1\}$, avoiding subwords $0i,(i=1,\ldots,b)$.
\end{proposition}
\begin{proof} We let $d(n)$ denote the number of the words of length $n-1$. Firstly, $d(0)=1$, since only the empty word has length $0$. Next, $d(1)=a$, since there are no restrictions on words of length $1$.
Assume that $n>1$. A word of length $n$ may begin with any letter. We have $a\cdot d(n-1)$ such words. From this number, we must subtract words which begin with subwords $0i,(i=1,2,\ldots,b)$.
Hence, $d(n)$ satisfies the same recurrence as $f_0(n)$, and the proposition is proved.
\end{proof}
\begin{example}
\begin{enumerate}
\item
If $a=2,b=1$, we have
\[f_0(1)=1,f_0(2)=2,f_0(n+2)=2f_0(n+1)-f_0(n),(n\geq 1),\]
which yields that $f_0(n)=n$. Hence, $n$ is the number of binary words of length $n-1$  avoiding subword $01$.
\item If $a=3,b=1$, we have
\[f_0(1)=1,f_0(2)=3,f_0(n+2)=3f_0(n+1)-f_0(n),(n\geq 1),\]
which is a well-known recurrence for the Fibonacci numbers $F_{2n}$. Hence,
$F_{2n}$ is the number of ternary words of length $n-1$ avoiding subword $01$.
\end{enumerate}
\end{example}
We now consider the particular case $a=b+1$.
\begin{corollary}
If $b>1$ and $a=b+1$, then
 \begin{equation*}f_0(n)=\frac{b^n-1}{b-1}.\end{equation*}
\end{corollary}
\begin{proof}
We have $f_0(1)=1,f_0(2)=1+b=a$. Further,
\begin{equation*}f_0(n+2)=\frac{b^{n+2}-1}{b-1}.\end{equation*}
On the other hand, we have
\begin{equation*}
(b+1)f_0(n+1)-bf_0(n)=(b+1)\cdot\frac{b^{n+1}-1}{b-1}-b\cdot\frac{b^n-1}{b-1}=\frac{b^{n+2}-1}{b-1}.
\end{equation*}
\end{proof}
In particular, for $a=3,b=2$, we have $f_0(n)=2^n-1$, which yields
\begin{corollary} The Mersenne number $2^n-1$ is the number of ternary words of length $n-1$ avoiding $01$ and $02$.
\end{corollary}
Using~\cite[Corollary 9]{ja1}, we obtain
\begin{equation}
\label{ab1}f_m(1)=1,f_m(2)=m+a;f_m(n+2)=(a+m)f_m(n+1)-bf_m(n),(n\geq 1).\end{equation}
This means that $f_m$ counts the same sort of words as $f_0$, with $m+a$ instead of $a$.

Using Proposition \ref{alf}, we obtain the following combinatorial interpretation of $c_m(n,k)$.
\begin{corollary} The number $c_m(n,k)$ is the  number of words of length $n-1$ over the alphabet $\{0,1,\ldots,b-1,b\ldots,m+a-1\}$ having exactly $k-1$  letters equal $m+a-1$ and avoiding subwords $0j,(j=1,2,\ldots,b)$.
\end{corollary}
We next derive an explicit formula for $c_1(n,k)$. A generating function for the sequence
$f_0(1),f_0(2),\ldots$ is $\frac{1}{bx^2-ax+1}$. According to~\cite[Equation (1)]{ja3}, we have \[\frac{x^k}{(bx^2-ax+1)^k}=\sum_{n=k}^\infty c_1(n,k)x^k.\]
 The numbers $\alpha=\frac{a+\sqrt{a^2-4b}}{2b}$ and $\beta=\frac{a-\sqrt{a^2-4b}}{2b}$ are the solutions of the equation $bx^2-ax+1=0$.
\begin{proposition}\label{p10} We have
\[c_1(n,k)=\frac{1}{b^k}\sum_{j=0}^{n-k}\frac{1}{\alpha^{j+k}\beta^{n-j}}{n-j-1\choose k-1}{k+j-1\choose k-1}.\]
\end{proposition}
\begin{proof}
We expand $\frac{x^k}{b^k(\alpha-x)^k(\beta-x)^k}$ into powers of $x$.
Since \[\frac{1}{(\gamma -x)^k}=\sum_{i=0}^\infty{k+i-1\choose k-1}\frac{x^i}{\gamma^{i+k}},\] we easily obtain
\[\frac{x^k}{b^k(\alpha-x)^k(\beta-x)^k}=\sum_{i=0}^\infty
\left[\sum_{j=0}^i\frac{1}{b^k\alpha^{j+k}
\beta^{i-j+k}}
{k+j-1\choose k-1}{k+i-j-1\choose k-1}\right]x^{i+k},\]
 and the statement follows by replacing $i$ by $n-k$.
 \end{proof}

In the case $a=b+1$, we have $\alpha=1$ and $\beta=\frac 1b$.
Therefore, the following formula holds:
\begin{equation}\label{amb}c_1(n,k)=\sum_{i=0}^{n-k}
b^{n-k-i}{n-i-1\choose k-1}{k+i-1\choose k-1}.\end{equation}
Using (\ref{cmnk}), we obtain
\begin{identity}
\begin{equation*}
\sum_{i_1+i_2+\cdots+i_k=n}\left[\prod_{t=1}^k(b^{i_t}-1)\right]=\sum_{i=0}^{n-k}
b^{n-k-i}(b-1)^k{n-i-1\choose k-1}{k+i-1\choose k-1},
\end{equation*}
where $i_t>0,(t=1,2,\ldots,k)$.
\end{identity}
\begin{remark} Using (\ref{cm1}) and  (\ref{suma1}), we  obtain explicit formulas  for $c_m(n,k)$ and $f_m(n)$.
\end{remark}
\section{Case 4}
 We let $\mathcal R$ denote the condition given in this case.
We first solve the problem for binary words.
\begin{proposition}Let $f_0(n)$ be the number of binary words satisfying $\mathcal R$. Then,
\begin{enumerate}
\item $f_0(1)=1,f_0(2)=0,f_0(n+2)=f_0(n+1)+f_0(n),(n>1).$
\item For $n>1$, we have  $f_0(n)=F_{n-2}$.
\end{enumerate}
\end{proposition}
\begin{proof}
 \begin{enumerate}
 \item
  We have $f_0(1)=1$, since the empty word has length $0$. Next, $f_0(2)=0$, since no words of length $1$ satisfy $\mathcal R$. Also, $f_0(3)=1$, since $10$ is the only word of length $2$ satisfying $\mathcal R$.
 Next, $f_0(4)=1$, since $100$ is the only word of length $3$ satisfying $\mathcal R$. Assume that $n>1$. Then,
    \begin{equation*}f_0(n+4)=f_0(n+2)+f_0(n+1)+\cdots,\end{equation*}
    since a word of length greater than $3$ must begin with a subword of the form $10\ldots 0$.
    Analogously, we obtain
\begin{equation*}f_0(n+5)=f_0(n+3)+f_0(n+2)+\cdots.\end{equation*}
Comparing these two  equations, we get
\begin{equation*}f_0(n+5)=f_0(n+4)+f_0(n+3).\end{equation*}
\item
The formula follows from the preceding recurrence.
\end{enumerate}
\end{proof}
Since $f_0(1)=1$, and so $f_m(1)=1$, using Proposition \ref{alf} and (\ref{suma1}), we obtain the following combinatorial interpretations of $f_m$ and $c_m(n,k)$.
\begin{corollary}
\begin{enumerate}
\item
The number $c_m(n,k)$ is the number of words over the alphabet
 $\{0,1,\ldots,m+1\}$ of length $n-1$ having $k-1$ letters equal $m+1$ and satisfying  $\mathcal R$.
\item  The number $f_m(n)$ is the number of words of length $n-1$ over
the alphabet $\{0,1,\ldots,m\}$ satisfying $\mathcal R$.
\end{enumerate}
\end{corollary}

We next derive an explicit formula for $c_1(n,k)$.
It is known that   $c_1(n,k)$ is the coefficient  of $x^n$ in the expansion of
$\left(\sum_{i=1}^\infty F_{i-2}x^i\right)^k$ into powers of $x$.

We consider the following auxillary initial function:
 \begin{equation*}\overline{f}_0(1)=0,\overline{f}_0(n)=1,(n>1).\end{equation*}
From~\cite[Proposition 23]{ja1}, we obtain $\overline{f}_1(n)=F_{n-1}$.
It is proved in~\cite[Proposition 13]{ja2} that
\[\overline{c}_2(n,k)={n-k-1\choose k-1},\left(k=1,2,\ldots,\left\lfloor\frac n2\right\rfloor\right),\]
and $\overline{c}_2(n,k)=0$ otherwise.

Using~\cite[Proposition 6]{ja3} yields
\[\overline{c}_3(n,k)=\sum_{i=k}^{\left\lfloor\frac n2\right\rfloor}{i-1\choose k-1}{n-i-1\choose i-1}.\]
Hence,
\begin{equation}\label{xy}\left(\sum_{i=1}^\infty F_{i-1}x^i\right)^k=\sum_{n=k}^\infty\overline{c}_3(n,k)x^n.\end{equation}
We let $X$ denote $\sum_{i=1}^\infty F_{i-1}x^i$.
 We have to expand the  expression $\left(\sum_{i=1}^\infty F_{i-2}x^i\right)^k,$
 which we denote by $Y$.
It follows that $Y=x^k(1+X)^k$. Hence,
\[Y=x^k\left(1+\sum_{i=1}^k{k\choose i}X^i\right)^k=\sum_{n=k}^\infty c_1(n,k)x^n.\]
Using  (\ref{xy}) yields
\[Y=x^k+\sum_{i=1}^k\sum_{j=i}^\infty{k\choose i}\overline{c}_3(j,i)x^{j+k}.\]
It is easy to see that, in the case $j+k=n$, the coefficient of $x^n$ on the right-hand side of this equation equals $\sum_{i=1}^k{k\choose i}\overline{c}_3(n-k,i)$.
We thus obtain
\begin{proposition}\label{p16} The following equations hold:
\begin{gather*}c_1(k,k)=1,\\
c_1(n,k)=\sum_{i=1}^k\sum_{j=i}^{\left\lfloor\frac{n-k}{2}\right\rfloor}{k\choose i}{j-1\choose i-1}{n-k-j-1\choose j-1},(n>k).\end{gather*}
\end{proposition}
Using~\cite[Corollary 9]{ja1}, we easily obtain the following recurrence for $f_m$:
\[f_m(1)=1,f_m(2)=m,f_m(n+2)=(m+1)f_m(n+1)-(m-1)f_m(n).\]

Some particular cases are of note. In the case $m=1$, we obtain
\[f_1(1)=1,f_1(2)=1,f_1(n+2)=2f_1(n+1),(n>1),\]
which implies
\[f_1(1)=f_1(2)=1,f_1(n)=2^{n-2},(n>2).\]
We thus obtain the following property of powers of $2$.
\begin{corollary}
For $n\geq 2$, the number $2^{n-2}$ is the number of ternary words of length $n-1$ satisfying $\mathcal R$.
\end{corollary}
It yields that the following Euler-type identity holds:
\begin{identity} For $n>2$, the number of binary words of length $n-2$ is the
 number of ternary words of length $n-1$, in which $0$ and $1$ appear only in a run of the
 form $1i$, where $i$ is the run of zeros of length $i\geq 1$.
\end{identity}
 From Propositions \ref{p16} and (\ref{suma1}), we obtain the following identity for the Mersen numbers:
\begin{identity}
\begin{equation*}2^{n-2}-1=\sum_{k=1}^{n}\sum_{i=1}^k\sum_{j=i}^{\left\lfloor\frac{n-k}{2}\right\rfloor}{k\choose i}{j-1\choose i-1}{n-k-j-1\choose j-1},(n>2).\end{equation*}
\end{identity}

We now consider the second particular case $m=2$. We have
\[f_2(1)=1,f_2(2)=2,f_2(n+2)=3f_2(n+1)-f_2(n),\] which is the recurrence
for Fibonacci numbers $F_{2n-1}$. We thus have
\begin{corollary}
The Fibonacci number $F_{2n-1}$ is the number of quaternary words of length $n-1$ satisfying $\mathcal R$.
\end{corollary}
Calculating values for $c_2(n,k)$, we obtain
\begin{identity}
\begin{equation*}F_{2n-1}=\sum_{k=1}^n\sum_{i=k}^{n}\sum_{t=0}^i\sum_{j=t}^{\left\lfloor\frac{n-i}{2}\right\rfloor}
{i-1\choose k-1}{i\choose t}{j-1\choose t-1}{n-i-j-1\choose j-1}.\end{equation*}
\end{identity}
\begin{remark} Using (\ref{cm1}) and  (\ref{suma1}),
 we  obtain the explicit formulas  for $c_m(n,k)$ and $f_m(n)$.
\end{remark}
\section{Case 5}
We let $\mathcal R$ denote the given condition. Again, we first consider the binary words.
 \begin{proposition}
 \begin{enumerate}
 \item The following formula holds:
 \begin{gather*}f_0(1)=1,f_0(2)=0,f_0(3)=1;\\
 f_0(n+3)=f_0(n+1)+f_0(n),(n\geq 1).\end{gather*}
\item We have $f_0(n)=p_{n+2}$, where $p_n$ is the $n$th Padovan number.
\end{enumerate}
\end{proposition}
\begin{proof} The first statement is easy to prove. Since $(1)$ is essentially
 the recurrence for the Padovan numbers, the statement $(2)$ is true.
\end{proof}
   This means that the Padovan
number $p_{n+2}$ is the number of binary words
of length $n-1$ in which $0$ appears in runs of even length, while $1$ appears in
runs, the lengths of which are divisible by  $3$. This means that the Padovan numbers count
the compositions into parts $2$ and $3$, which is a well-known.
\begin{corollary}\label{fib}
\begin{enumerate}
\item
 The function $f_m$ satisfies the following recurrence:
\begin{gather*}f_m(1)=1,f_m(2)=m,f_m(3)=m^2+1,\\f_m(n+3)=mf_m(n+2)+f_m(n+1)+f_m(n),(n>1).\end{gather*}
\item Then, $f_m(n)$ is the number of words of length $n-1$ over $\{0,1,\ldots,m+1\}$ satisfying $\mathcal R$.
\item  Also, $c_m(n,k)$ is the number of words of length $n-1$ over $\{0,1,\ldots,m+1\}$ having $k-1$ letters equal to $m+1$, and satisfying $\mathcal R$.
\end{enumerate}
\end{corollary}
\begin{proof}
 The claim $(1)$ easily follows from~\cite[Theorem 6]{ja1}.
The claims $(2)$ and $(3)$ follow from Proposition \ref{alf}.

We add a short combinatorial proof for $(2)$.
Equation $f_m(1)=1$ means that the empty word satisfies $\mathcal R$. Further, $f_m(2)=m$
means that a word of length $1$ may consist of any letter except $0$ and $1$.
 Next, $f_m(3)=m^2+1$ means that a word of length $2$ may consist of pairs
  from $\{2,3,\ldots,m+1\}$, which are $m^2$ in number, plus the word $00$.
Finally, a word of length $n>2$ may begin with any letter from  $\{2,3,\ldots,m+1\}$,
 or from $00$, or from $111$.
\end{proof}
The case $m=1$ in Corollary \ref{fib} is the recurrence for Tribonacci numbers. Hence,
\begin{corollary} The sequence $1, 1, 2, 4, 7,\ldots$ of the Tribonacci numbers is the invert transform of the sequence  $1,0,1,1,1,2,\ldots$ of the Padovan numbers.

 Also, Tribonacci numbers count ternary words satisfying $\mathcal R$.
\end{corollary}
Finally, we calculate $c_1(n,k)$.
We define the arithmetic function $\overline{f}_0$ such that $\overline{f}_0(2)=\overline{f}_0(3)=1$, and $\overline{f}_0(n)=0$ otherwise. It follows from~\cite[Propositon 5]{ja2}  that $\overline{c}_1(n,k)={k\choose n-2k}.$ Also, using~\cite[Theorem 6]{ja1}, we obtain
 \begin{gather*}\overline{f}_1(1)=0,\overline{f}_1(2)=1,\overline{f}_1(3)=1,\\
 \overline{f}_1(n+3)=\overline{f}_0(n+1)+\overline{f}_0(n).\end{gather*}
This implies that $\overline{f}_1(n)=f_0(n-1),(n>1)$. The sequence $f_0(1),f_0(2),\ldots$ is thus obtained by inserting $1$ at the beginning of the sequence
$\overline{f}_1(1),\overline{f}_1(2),\ldots$.

Using~\cite[Equation (10)]{ja3}, we obtain
\begin{equation*}\overline{c}_2(n,k)=\sum_{i=k}^n{i-1\choose k-1}\cdot{i\choose n-2\cdot i}.\end{equation*}
On the other hand,~\cite[Proposition 2]{ja3} yields
\begin{equation}\label{jjj}\left(\sum_{i=1}^\infty\overline{f}_1(i)x^i\right)^k=
\sum_{n=k}^\infty\overline{c}_2(n,k)x^n.\end{equation}

To obtain an explicit formula for $c_1(n,k)$, we need to expand the expression $X$ given by
$X=\left(\sum_{i=1}^\infty f_0(i)x^i\right)^k$ into powers of $x$.
We have \begin{equation*}X=\left(x+\sum_{i=2}^\infty f_0(i)x^i\right)^k=(x+xY)^k,\] where
 $Y=\sum_{i=1}^\infty \overline{f}_1(i)x^i$.
Hence, \begin{equation*}X=x^k\sum_{i=0}^k{k\choose i}Y^i.\end{equation*}
Applying Equation(\ref{jjj}) yields
\begin{equation*}X=\sum_{i=0}^k{k\choose i}\sum_{j=i}^\infty\overline{c}_2(j,i)x^{j+k}.\end{equation*}
Taking $n=j+k$, we get
\begin{proposition} The following formula holds:
\begin{equation*}c_1(n,k)=\sum_{i=0}^k\sum_{j=i}^{n-k}{k\choose i}{j-1\choose i-1}{j\choose n-k-2j}.\end{equation*}
\end{proposition}
We thus obtain the following identity for the Tribonacci numbers $T_n$:
\begin{identity}
\begin{equation*} T_n=\sum_{k=1}^n\sum_{i=0}^k\sum_{j=i}^{n-k}{k\choose i}{j-1\choose i-1}{j\choose n-k-2j}.\end{equation*}
\end{identity}
\begin{remark} Using (\ref{cm1}) and  (\ref{suma1}), we  obtain explicit formulas  for $c_m(n,k)$ and $f_m(n)$.
\end{remark}

\end{document}